\documentclass[12pt, reqno]{amsart}
\usepackage{latexsym,amsfonts,amsmath,amssymb,color,amsthm}

\usepackage{enumerate}

\usepackage[colorlinks, bookmarks=true]{hyperref}
\usepackage{color,graphicx,shortvrb}
\usepackage[active]{srcltx} 

\newtheorem{theorem}{Theorem}[section]
\newtheorem{remark}[theorem]{Remark}

\newtheorem{lemma}[theorem]{Lemma}

\newtheorem{proposition}[theorem]{Proposition}
\newtheorem{corollary}[theorem]{Corollary}

\newtheorem*{theorem*}{Theorem (Denjoy-Wolff)}
\newtheorem*{theoremA*}{Theorem A}
\newtheorem*{theoremB1*}{Theorem B1}
\newtheorem*{theoremB2*}{Theorem B2}
\newtheorem*{theoremC*}{Theorem C}

\catcode`\@=11

\renewcommand{\section}%
   {\setcounter{equation}{0}\@startsection {section}{1}{\z@}{-3.5ex plus -1ex
  minus -.2ex}{2.3ex plus .2ex}{\Large\bf}}

\numberwithin{equation}{section}

\newcommand{\C}{\mathbb{C}} 

\usepackage[latin 1]{inputenc}

\title{Stability of derivations under weak-2-local continuous perturbations}

\date{\today}

\author[E. Jord\'a]{Enrique Jord\'a}

\address{Escuela Polit\'ecnica Superior de Alcoy, IUMPA, Universitat Polit\'ecnica de Valencia, Plaza Ferr\'andiz y Carbonell 1, 03801 Alcoy, Spain}
\email{ejorda@mat.upv.es}

\author[A.M. Peralta]{Antonio M. Peralta}

\address{Departamento de An{\'a}lisis Matem{\'a}tico, Universidad de Granada,
Facultad de Ciencias 18071, Granada, Spain}
\email{aperalta@ugr.es}

\thanks{First author partially supported by the Spanish Ministry of Economy and Competitiveness Project MTM2013-43540-P. Second author partially supported by the Spanish Ministry of Economy and Competitiveness and European Regional Development Fund project no. MTM2014-58984-P and Junta de Andaluc\'{\i}a grant FQM375.}



\begin{document}
\begin{abstract} Let $\Omega$ be a compact Hausdorff space and let $A$ be a C$^*$-algebra. We prove that if every weak-2-local derivation on $A$ is a linear derivation and every derivation on $C(\Omega,A)$ is inner, then every weak-2-local derivation $\Delta:C(\Omega,A)\to C(\Omega,A)$ is a {\rm(}linear{\rm)} derivation. As a consequence we derive that, for every complex Hilbert space $H$, every weak-2-local derivation $\Delta : C(\Omega,B(H)) \to  C(\Omega,B(H))$ is a (linear) derivation. We actually show that the same conclusion remains true when $B(H)$ is replaced with an atomic von Neumann algebra. With a modified technique we prove that, if $B$ denotes a compact C$^*$-algebra (in particular, when $B=K(H)$), then every weak-2-local derivation on $C(\Omega,B)$ is a (linear) derivation. Among the consequences, we show that for each von Neumann algebra $M$ and every compact Hausdorff space $\Omega$, every 2-local derivation on $C(\Omega,M)$ is a {\rm(}linear{\rm)} derivation.
\end{abstract}

\maketitle

\section{Introduction}

We recall that a derivation from a Banach algebra $A$ into a Banach $A$-bimodule $X$ is a linear map $D: A\to X$ satisfying \begin{equation}\label{eq Leibniz rule} D(a b) = D(a) b + a D(b),
\end{equation} for every $a, b$ in $A$. Given $x_0\in X$, the operator $\hbox{ad}_{x_0}  : A \to X$, $a\mapsto \hbox{ad}_{x_0} (a) = [x_0,a ] = x_0 a - a x_0,$ is a derivation. Derivations of this form are termed \emph{inner derivations}. A linear mapping $T: A\to X$ is called a \emph{local derivation} if for each $a\in A$ there exits a derivation $D_{a}: A\to X$, depending on $a$, satisfying $T(a) = D_a(a)$.\smallskip

In the setting above, the dual space $X^*$ can be equipped with a natural structure of Banach $A$-bimodule with respect to the products $$ (a \phi) (x) = \phi (x a),\hbox{ and, } (\phi a) (x) = \phi (ax) \ \ (a\in A, x\in X, \phi\in X^*).$$

Derivations whose domain is a C$^*$-algebra or a von Neumann algebra are, by far, the most studied and best understood class of derivations. S. Sakai sets some of the most influencing results by showing that every derivation on a C$^*$-algebra is automatically continuous (see \cite{Sak60} or \cite[Lemma 1.4.3]{S}), and every derivation on a von Neumann algebra or on a unital simple C$^*$-algebra is inner (cf. \cite[Theorems 4.1.6 and 4.1.11]{S}). There are examples of derivations on a C$^*$-algebra which are not inner (see \cite[Example 1.4.8]{S}). Subsequently, J. Ringrose established in  \cite{Ringrose72} that every derivation from a C$^*$-algebra $A$ into a Banach $A$-bimodule is continuous. A derivation $D$ on a C$^*$-algebra $A$ is called a \emph{$^*$-derivation} if $D(a^*) = D(a)^*$ for all $a\in A$. An inner derivation $\hbox{ad}_{x_0}: A\to A$ is a $^*$-derivation if and only if $x_0^* = -x_0$. Let $\Delta$ be a mapping from a C$^*$-algebra $A$ into a C$^*$-algebra $B$. $\Delta^{\sharp} : A\to B$ is the map defined by  $\Delta^{\sharp} (a) =\Delta(a^*)^*$ ($a\in A$). $\Delta$ is called \emph{symmetric} if $\Delta^{\sharp} = \Delta$, i.e.,  $\Delta (a^* ) = \Delta (a)^*$, for every $a\in A$. \smallskip

Researchers belonging to different generations have been striving to explore the stability of the set of derivations under weaker and weaker hypothesis since 1990. The pioneering contribution of R.V. Kadison in \cite{Kad90} asserts that every continuous local derivation from a von Neumann algebra $M$ into a dual Banach $M$-bimodule $X$ is a derivation.
B.E. Johnson shows in \cite{John01} that the conclusion in Kadison's theorem holds for local derivations from a C$^*$-algebra $A$ into a Banach $A$-bimodule. These results can be interpreted as properties of stability of derivations under local perturbations in the set of linear maps from $A$ into $X$.\smallskip

A weaker stability property derives from the notion of 2-local derivation. Let $X$ be a Banach $A$-bimodule over a Banach algebra $A$. Following P. \v{S}emrl \cite{Semrl97}, a mapping $\Delta : A\to X$ is said to be a \emph{2-local derivation} if for every $a,b\in A$, there exists a derivation $D_{a,b}: A\to X$, depending on $a$ and $b$, satisfying $\Delta(a) = D_{a,b}(a)$ and $\Delta(b) = D_{a,b}(b)$. In a recent contribution, Sh. Ayupov and K. Kudaybergenov prove that every 2-local derivation on a von Neumann algebra $M$ is a derivation, that is, derivations are stable under 2-local perturbations in the set of mappings on $M$ (see \cite{AyuKuday2014}). As long as we know the case of 2-local derivations on a  general C$^*$-algebra remains as an open problem.\smallskip

More recent studies are leading the mathematical community to the notion of weak-2-local $\mathcal{S}$-maps between Banach spaces (see \cite{NiPe2014,NiPe2015,CaPe2015,CaPe2016}). According to the notation in the just quoted references, given a subset $\mathcal{S}$ of the space $L(X,Y)$, of all linear maps between Banach spaces $X$ and $Y$, a (non-necessarily linear nor continuous) mapping $\Delta : X\to Y$ is said to be a \emph{weak-2-local $\mathcal{S}$ map} (respectively, a \emph{2-local $\mathcal{S}$-map}) if for every $x,y\in X$ and $\phi\in Y^{*}$ (respectively, for every $x,y\in X$), there exists $T_{x,y,\phi}\in \mathcal{S}$, depending on $x$, $y$ and $\phi$ (respectively, $T_{x,y}\in \mathcal{S}$, depending on $x$ and $y$), satisfying $$\phi \Delta(x) = \phi T_{x,y,\phi}(x), \hbox{ and  }\phi \Delta(y) = \phi T_{x,y,\phi}(y)$$ (respectively, $ \Delta(x) =  T_{x,y}(x),$ and $  \Delta(y) = T_{x,y}(y)$). If we take $\mathcal{S}=K(X,Y)$ the space of compact linear mappings from $X$ to $Y$ then it is straightforward to check that if $\Delta$ is any non-linear, 1-homogeneous map, i.e. $\Delta(\alpha x)=\alpha \Delta(x)$ for each $\alpha\in\C$, then $\Delta$ is a 2-local $\mathcal{S}$-map. Some particular cases receive special names. When $\mathcal{S}$ is the set of all derivations on a C$^*$-algebra $A$ (respectively, the set of all $^*$-derivations on $A$, or, more generally, the set of all symmetric maps from $A$ into a C$^*$-algebra $B$), weak-2-local $\mathcal{S}$-maps are called \emph{weak-2-local derivations} (respectively, \emph{weak-2-local $^*$-derivations} or \emph{weak-2-local symmetric maps}). \smallskip

The first results in this line prove that every weak-2-local derivation on a finite dimensional C$^*$-algebra is a linear derivation (see \cite[Corollary 2.13]{NiPe2015}), and for a separable complex Hilbert space $H$ every {\rm(}non-necessarily linear nor continuous{\rm)} weak-2-local $^*$-derivation on $B(H)$ is linear and a $^*$-derivation \cite[Theorem 3.10]{NiPe2015}. Symmetric maps between C$^*$-algebras $A$ and $B$ are very stable under weak-2-local perturbations in the space $L(A,B)$, more concretely, every weak-2-local symmetric map between C$^*$-algebras is a linear map \cite[Theorem 2.5]{CaPe2015}. Fruitful consequences are derived from this result, for example, every weak-2-local $^*$-derivation on a general C$^*$-algebra is a (linear) $^*$-derivation, and  every 2-local $^*$-homomorphism between C$^*$-algebras is a (linear) $^*$-homomorphism (see \cite[Corollary 2.6]{CaPe2015}). Weak-2-local derivations on von Neumann algebras and on general C$^*$-algebras remain unknowable. In \cite{CaPe2016}, J.C. Cabello and the second author of this note prove that every weak-2-local derivation on $B(H)$ or on $K(H)$ is a linear derivation, where $H$ is an arbitrary complex Hilbert space, and consequently, every weak-2-local derivation on an atomic von Neumann algebra or on a compact C$^*$-algebra is a linear derivation.\smallskip

Another attempt to study 2-local derivations on C$^*$-algebras has been conducted by Sh. Ayupov and F.N. Arzikulov \cite{AyuArz}. The main result in the just quoted paper proves that, for every compact Hausdorff space $\Omega$, every 2-local derivation on $C(\Omega,B(H))$ is a derivation.\smallskip

In this note we continue with the study of weak-2-local derivations in new classes of C$^*$-algebras. In the first main result (Corollary \ref {c weak-2-local der on finite dimensional and atomic}) we prove that, for every Hilbert space $H$, every weak-2-local $\Delta:C(\Omega,B(H))\to C(\Omega,B(H))$ is a {\rm(}linear{\rm)} derivation. This result is a consequence of a technical theorem in which we establish that if $A$ is a C$^*$-algebra such that every weak-2-local derivation on $A$ is a linear derivation and every derivation on $C(\Omega,A)$ is inner, then every weak-2-local derivation $\Delta:C(\Omega,A)\to C(\Omega,A)$ is a {\rm(}linear{\rm)} derivation (see Theorem \ref{t weak-2-local der on CKA}). Actually this technical result, combined with recent results on weak-2-local derivations on atomic von Neumann algebras in \cite{CaPe2016}, implies that every weak-2-local $C(\Omega,A)$ is a linear derivation whenever $A$ is an atomic von Neumann algebra. We particularly show that the result of Sh. Ayupov and F.N. Arzikulov remains true if we replace ``2-local'' with ``weak-2-local''.\smallskip

In order to extend our study to weak-2-local derivations on C$^*$-algebras of the form $C(\Omega,K(H))$, where $K(H)$ is the C$^*$-algebra of compact linear operators on a Hilbert space $H$, or to $C(\Omega,B)$ where $B$ is a compact C$^*$-algebra, we represent, in Proposition \ref{p derivations on C(K,K(H))}, every derivation on $C(\Omega, K(H))$ as an ``inner derivations'' associated with a mapping $Z_0 : \Omega \to B(H)$ which is, in general, $\tau$-weak$^*$-continuous, where $\tau$ is the topology of $\Omega$. The $\tau$-norm continuity of the mapping $Z_0$ cannot be, in general, pursued (compare Remark \ref{new remark}). The paper culminates with a result asserting that, for a compact C$^*$-algebra $B$, every weak-2-local derivation $\Delta:C(\Omega,B)\to C(\Omega,B)$ is a {\rm(}linear{\rm)} derivation (see Theorem \ref{t weak-2-local der on compact Cstar}).

\section{Weak-2-local derivations on $C(\Omega)\otimes A$}

Let $A$ and $B$ be C$^*$-algebras. Henceforth $A\odot B$ will denote the algebraic tensor product of $A$ and $B$. A norm $\alpha$ on $A\odot B$ is said to be a C$^*$-norm if $\alpha (x x^*) = \alpha (x)$ and $\alpha (x y) \leq \alpha (x) \alpha (y)$, for every $x,y\in A\odot B$. It is known that there exists a least C$^*$-norm $\alpha_0$ among all C$^*$-norms $\alpha$ on $A \odot B$ such that $\alpha^*$ is finite. It is further known that $\alpha_0$ is a cross norm and $\lambda \leq \alpha_0 \leq \gamma$, where $\lambda$ and $\varepsilon$ denote the injective and the projective tensor norm on $A\odot B$, respectively (see \cite[Propsition 1.22.2]{S}). Along this note, the symbol $A\otimes B$ will denote the C$^*$-algebra obtained as the completion of $A\odot B$ with respect to $\alpha_0$.\smallskip

For each C$^*$-algebra $A$ and every locally compact Hausdorff space $L$, the Banach space $C_0(L,A)$, of all $A$-valued continuous functions on $L$ vanishing at infinite, admits a natural structure of C$^*$-algebra with respect to the ``pointwise'' operations and the $\sup$ norm. It is also known that $C_0 (L,A)$ and $C_0(L)\otimes A = C_0(L)\otimes_{\lambda} A$ are isometrically C$^*$-isomorphic as C$^*$-algebras (see \cite[Proposition 1.22.3]{S}). When $\Omega$ is a compact Hausdorff space, we have $C (\Omega,A)\cong C(\Omega)\otimes A = C(\Omega)\otimes_{\lambda} A$.\smallskip

Another influential result due to S. Sakai, apart from those commented in the introduction, asserts that every von Neumann algebra admits a unique (isometric) predual and its product is separately weak$^*$-continuous (see \cite[Theorem 1.7.8]{S}). It is known that, for a C$^*$-algebra $A$, its second dual, $A^{**}$, is a von Neumann algebra \cite[Theorem 1.17.2]{S}. Combining these facts with the identity in \eqref{eq Leibniz rule}, we can see that for every derivation $D: A\to A$, its bitransposed map $D^{**}: A^{**} \to A^{**}$ is a derivation on $A^{**}$. Therefore, there exists $z_0$ in $A^{**}$ satisfying $D^{**} (x) = [z_0, x]$, for every $x\in A^{**}$. We have already commented that we cannot assume that $z_0$ lies in $A$. The following lemma, which was originally proved by R.V. Kadison in \cite[Theorem 2]{Kad66}, can be also derived from the above results:

\begin{lemma}\label{l derivaiton annihilates the center}\cite[Theorem 2]{Kad66} Let $D: A\to A$ be a derivation on a C$^*$-algebra $A$ whose center is denoted by $Z(A)$. Then $D (c) =0$, for every $c\in Z(A)$.$\hfill\Box$
\end{lemma}

A Banach algebra $A$ is said to be \emph{super-amenable} if every derivation from $A$ into a Banach $A$-bimodule is inner. The Banach algebra $A$ is called \emph{amenable} if for every Banach $A$-bimodule $X$, every derivation from $A$ into $X^*$ is inner. Finally if every derivation from $A$ into $A^*$ is inner we say that $A$ is weakly amenable. Every super-amenable (respectively, amenable) Banach algebra is amenable (respectively, weakly amenable), and the three classes are mutually different. We refer to the monographs \cite{John72,Run} for a detailed account on the theory of super-amenable, amenable and weakly amenable Banach algebras.\smallskip

The problem of determining those C$^*$-algebras admitting only inner derivations has been considered by a wide number of researchers. Although there are some basic unanswered questions along these lines, a partial result obtained by C.A. Akemann and B.E. Johnson (see \cite{AkJohn79}) will be very useful for our purposes.

\begin{theorem}\label{l C(KvN) has the inner derivation property or rigidity property}\cite[Theorem 2.3]{AkJohn79} Let $M$ be a von Neumann algebra, and let $B$ be a unital abelian C$^*$-algebra. Then every derivation of the C$^*$-tensor product $B \otimes M$ is inner. Equivalently, given a compact Hausdorff space $\Omega$, every derivation on  $C(\Omega,M)$ is inner.$\hfill\Box$
\end{theorem}

Let $X$ and $Y$ be Banach $A$-bimodules over a Banach algebra $A$. A linear mapping $\Phi: X\to Y$ is called a \emph{module homomorphism} if $\Phi (a x) = a \Phi (x)$ and $\Phi (x a) = \Phi (x) a$, for every $a\in A$, $x\in X.$ The next lemma, whose proof is left to the reader, gathers some basic properties.

\begin{lemma}\label{l composition module hom} Let $A$ be a Banach algebras and let $\Phi: X \to Y$ be a module homomorphism. \begin{enumerate}[$(a)$]
\item If $D: A\to X$ is a derivation, then $\Phi D : A \to Y$ is a derivation;
\item If $\Delta: A\to X$ is a weak-2-local derivation and $\Phi$ is continuous, then $\Phi \Delta : A \to Y$ is a weak-2-local derivation;
\item Suppose $B$ is another Banach algebra such that $X$ is a Banach $B$-bimodule and there is a homomorphism $\Psi : B\to A$ satisfying $\Psi (b) x = b x$ and $x b = x \Psi (b)$, for every $x\in X$, $b\in B$. Then for each derivation {\rm(}respectively, for each weak-2-local derivation{\rm)} $D: A\to X$ the composition $D \Psi: B \to X$ is a derivation {\rm(}respectively, a weak-2-local derivation{\rm)}.$\hfill\Box$
\end{enumerate}
\end{lemma}

Following standard notation, given $t\in \Omega$, $\delta_{t} : C(\Omega,A) \to A$ will denote the $^*$-homomorphism defined by $\delta_t (X) = X(t)$. The space $C(\Omega,A)$ also is a Banach $A$-bimodule with products $(a X) (t) = a X(t)$ and $(X a) (t) = X(t) a$, for every $a\in A$, $X\in C(\Omega,A)$. The mapping $\delta_{t} : C(\Omega,A) \to A$ is an $A$-module homomorphism.\smallskip

Given a compact Hausdorff space $\Omega$ and a C$^*$-algebra $A$, the $^*$-homomorphism mapping each element $a$ in $A$ to the constant function $\Omega\to A$, $t\mapsto a$ will be denoted by $\widehat{a}$. The mapping $\Gamma: A \to C(\Omega, A)= C(\Omega) \otimes A$, $\Gamma (a)= 1\otimes a = \widehat{a}$,  is an $A$-module homomorphism.

\begin{theorem}\label{t weak-2-local der on CKA} Let $\Omega$ be a compact Hausdorff space and let $A$ be a C$^*$-algebra. Suppose that every weak-2-local derivation on $A$ is a linear derivation, every derivation on $C(\Omega,A)$ is inner. Then every weak-2-local derivation $\Delta:C(\Omega,A)\to C(\Omega,A)$ is a {\rm(}linear{\rm)} derivation.
\end{theorem}

\begin{proof} By \cite[Theorem 3.4]{EssaPeRa14} it is enough to prove that $\Delta$ is linear. This linearity certainly holds if and only if $\delta_t \Delta$ is linear for all $t\in \Omega$. 
\smallskip

Fix an arbitrary $t\in \Omega$. We claim that \begin{equation}\label{eq compositon delta t} \delta_t \Delta \Gamma \delta_t (X)= \delta_t \Delta (X),
\end{equation} for every $X\in C(\Omega,A)$, where for each $a\in A$, $\Gamma (a) = \widehat{a}$ is the constant function with value $a$ on $\Omega$. Indeed, by hypothesis, every derivation $D: C(\Omega,A)\to C(\Omega,A)$ is inner, and hence of the form $D(X) = [Z,X]$ ($\forall X\in C(\Omega,A)$), where $Z\in C(\Omega,A)$. Thus, for each $t\in \Omega$ we have $\delta_t D \Gamma \delta_t (X)= \delta_t D (X)$, for every $X\in C(\Omega,A)$. Let $\phi$ be any element in $A^*$ and consider the functional $\phi\otimes \delta_t : C(\Omega,A)\to \mathbb{C}$, $X\to \phi(X(t))$. By the weak-2-local property of $\Delta$, for each $X\in C(\Omega,A)$, there exists a derivation $D=$ $D_{X,\Gamma(X(t)),\phi\otimes \delta_t}: C(\Omega,A)\to C(\Omega,A)$, depending on $X$, $\Gamma(X(t))=\Gamma \delta_t (X),$ and $\phi\otimes \delta_t$, such that $$ \phi \left( \delta_t \Delta (X) -  \delta_t \Delta \Gamma \delta_t (X)\right)= \phi\otimes \delta_t \left( \Delta (X) - \Delta \Gamma \delta_t (X)\right) = \phi\otimes \delta_t \left( D (X) - D \Gamma \delta_t (X)\right) = 0,$$ because $\delta_t D \Gamma \delta_t (X)= \delta_t D (X)$. The arbitrariness of $\phi\in A^*$ proves \eqref{eq compositon delta t}.\smallskip

Since the operators $\delta_{t} : C(\Omega,A) \to A$ is a continuous $A$-module homomorphism, and $\Gamma : A\to C(\Omega,A)$ is a homomorphism satisfying $\Gamma (a) A = a A$ and $A \Gamma(a) = A a$, for every $a\in A$, $A\in C(\Omega,A)$, Lemma \ref{l composition module hom}$(c)$ implies that  $\delta_t \Delta \Gamma : A\to A$ is a weak-2-local derivation for every $t\in \Omega$. The additional hypothesis on $A$ assure that $\delta_t \Delta \Gamma $ is a linear derivation. Therefore $$ \delta_t \Delta \Gamma (X(t)+Y(t)) = \delta_t \Delta \Gamma (X(t)) +\delta_t \Delta \Gamma (Y(t)),$$ for every $X,Y$ in $C(\Omega,A)$. By \eqref{eq compositon delta t} $$\delta_t \Delta (X+Y) = \delta_t \Delta \Gamma \delta_t (X+Y) = \delta_t \Delta \Gamma \delta_t (X) +\delta_t \Delta \Gamma \delta_t (Y) = \delta_t \Delta (X)+ \delta_t \Delta (Y),$$ for every $t\in \Omega$, $X,Y\in C(\Omega,A)$. In particular, $\delta_t \Delta$ is a linear mapping, as we desired.
\end{proof}

The previous theorem can be now applied to provide new non-trivial examples of C$^*$-algebras on which every weak-2-local derivation is a derivation.

\begin{corollary}\label{c weak-2-local der on finite dimensional and atomic} Let $H$ be a complex Hilbert space. Then every weak-2-local derivation $\Delta:C(\Omega,B(H))\to C(\Omega,B(H))$ is a {\rm(}linear{\rm)} derivation. Furthermore, for an atomic von Neumann algebra $A$, every weak-2-local derivation $\Delta:C(\Omega,A)\to C(\Omega,A)$ is a {\rm(}linear{\rm)} derivation.
\end{corollary}

\begin{proof} Theorem \ref{l C(KvN) has the inner derivation property or rigidity property} proves that every derivation on $C(\Omega,B(H))$ is inner. It is also known that every weak-2-local derivation on $B(H)$ is a linear derivation (see \cite[Theorem 3.1]{CaPe2016}). Theorem \ref{t weak-2-local der on CKA} implies that every derivation on $C(\Omega,B(H))$ is a linear derivation.\smallskip

The statement for atomic von Neumann algebras follows from the same arguments but replacing  \cite[Theorem 3.1]{CaPe2016} with  \cite[Corollary 3.5]{CaPe2016}.
\end{proof}

We observe that Corollary \ref{c weak-2-local der on finite dimensional and atomic} provides a generalization of a recent result due to Sh. Ayupov and F.N. Arzikulov (compare \cite[Theorem 1]{AyuArz}).

\begin{corollary}\label{c Ayupov Arzikulov}\cite[Theorem 1]{AyuArz} Let $H$ be a complex Hilbert space. Then every 2-local derivation $\Delta:C(\Omega,B(H))\to C(\Omega,B(H))$ is a {\rm(}linear{\rm)} derivation.$\hfill\Box$
\end{corollary}

We observe that for an infinite dimensional separable complex Hilbert space $H$, we can always find a derivation on the C$^*$-algebra $K(H)$ of all compact operators on $H$ which is not inner (see \cite[Example 4.1.8]{S}). Since the a similar conclusion remains valid for $C(\Omega,K(H))$, we cannot apply Theorem \ref{t weak-2-local der on CKA} in this case. We shall see next how to avoid the difficulties.

Throughout the paper, $M_n = M_n(\mathbb{C})$ will denote the complex $n\times n$-matrices. For each $i,j\in\{1,\ldots,n\}$, $e_{ij}$ will denote the unit matrix in $M_n$ with $1$ in the $(i,j)$ component and zero otherwise. Given a C$^*$-algebra $A$, the symbol $M_n (A)$ will stand for the $n\times n$-matrices with entries in $A$. It is known that $M_n (A)$ is a C$^*$-algebra with respect to the product and involution defined by $\displaystyle (a_{ij})_{i,j} (b_{ij})_{i,j} = \left(\sum_{k=1}^{n} a_{ik} b_{kj} \right)_{i,j}$ and $(a_{ij})_{i,j}^{*} = (a_{ji}^*)_{i,j}$, respectively (compare \cite[\S IV.3]{Tak}). The space $M_n (A)$ also is a Banach $A$-bimodule for the products $ b (a_{ij}) = (b a_{ij}),$ and $(a_{ij}) b = ( a_{ij} b)$. Given $a\in A$ and $i,j\in \{1,\ldots, n\}$, the symbol $a\otimes e_{ij}$ will denote the matrix in $M_n(A)$ with entry $a$ in the $(i,j)$-position and zero otherwise.\smallskip

The following lemma might be known, it is included here due to the lack of an explicit reference.

\begin{lemma}\label{l weakstar times norm convergent} Let $(z_{\lambda})$ and $(x_{\lambda})$ be bounded nets in a von Neumann algebra $M$ such that $(z_{\lambda})\to z_0$ in the weak$^*$-topology of $M$ and $(x_{\lambda})\to x_0$, in the norm topology of $M$. Then $(z_{\lambda} x_{\lambda})\to z_0 x_0$ in the weak$^*$-topology.
\end{lemma}

\begin{proof} We can assume that $\|z_{\lambda}\|,\|x_{\lambda}\|\leq 1$, for every $\lambda$. The net $(x_{\lambda}- x_0 )\to 0$ in norm. For each norm-one positive normal functional $\phi\in M_*$, it follows from the Cauchy-Schwarz inequality that $$| \phi (z_{\lambda} (x_{\lambda}- x_0 ))|^{2} \leq \phi (z_{\lambda} z_{\lambda}^*) \phi ((x_{\lambda}- x_0 )^* (x_{\lambda}- x_0 )) \leq \|x_{\lambda}- x_0\|^2 \to 0.$$ We deduce that $ (z_{\lambda} (x_{\lambda}- x_0 ) )\to 0$ in the weak$^*$-topology of $M$. Since the product of $M$ is separately weak$^*$-continuous, we also know that $(z_{\lambda} x_0 )\to z_0 x_0$ in the weak$^*$-topology, and hence $(z_{\lambda} x_{\lambda})\to z_0 x_0$ in the weak$^*$-topology.
\end{proof}

Suppose $\hbox{ad}_{x_0}  : A \to A$, $a\mapsto  [x_0,a ]$ is an inner derivation on a C$^*$-algebra. It is well known that the element $x_0$ is not uniquely determined by  $\hbox{ad}_{x_0}$, for example $\hbox{ad}_{x_0} = \hbox{ad}_{y_0}$ as derivations on $A$ if and only if $x_0-y_0\in Z(A)$.\smallskip

Concerning norms, it is easy to see that $\| [x_0, . ] \| \leq 2 \|x_0\|$, where $\| [x_0, . ] \|$ denotes the norm of the linear derivation in $B(A)$. It is not obvious that an element in the set $x_0+Z (A)$ can be bounded by a multiple of the norm of the inner derivation $[x_0, . ]$. In this line, R.V. Kadison, E.C. Lance and J.R. Ringrose prove, in \cite[Theorem 3.1]{KadLanRing47}, that for each $^*$-derivation $D$ on a C$^*$-algebra $A$, if $\widetilde{D}$ denotes its unique extension to a derivation on $A^{**}$, then there is a unique self-adjoint element $a_0$ in $A^{**}$ such that $\widetilde{D} = \hbox{ad}_{a_0}=[a_0,.]$ and, for each central projection $q$ in $A^{**}$, we have $\| a_0 q\| = \frac12 \| \widetilde{D}|_{A^{**} q}\|$. In particular $\|D\| = 2 \|a_0\|$.\smallskip

Let $D: A\to A$ be a derivation on a C$^*$-algebra. We can write $D = D_1 + i D_2$, where $D_1 = \frac12 (D+D^{\sharp})$, $D_2 = \frac{1}{2 i} (D-D^{\sharp})$ are $^*$-derivations on $A$ with $\|D_{j}\| \leq \|D\|$, for every $j\in\{1,2\}$. Let $\widetilde{D}, \widetilde{D}_j : A^{**} \to A^{**}$ denote the unique extension of $D$ and $D_j$ to a derivation on $A^{**}$, respectively. Applying the just quoted result, we find $a_0,b_0\in A^{**}_{sa}$ satisfying $\widetilde{D}_1 = \hbox{ad}_{a_0}=[a_0,.]$, $\widetilde{D}_2 = \hbox{ad}_{b_0}=[b_0,.]$ and, $\|a_0\|, \|b_0\|\leq \frac12 \|D\|$. Then $\widetilde{D}  (x) = [a_0 + i b_0, x]$, for every $x\in A^{**}$, with $\|a_0+ i b_0 \|\leq \|D\|$.\smallskip

In general, $C(\Omega,M_n)= C(\Omega)\otimes M_n$ need not be a von Neumann algebra and its second dual is too big for our purposes. We have already commented that every derivation $D:C(\Omega,M_n) \to C(\Omega,M_n)$ is inner, so there exists $X_0\in C(\Omega,M_n)$ satisfying $D(X) = [X_0,X]$, for every $X\in C(\Omega,M_n)$. If $D$ is a $^*$-derivation we can assume that $X_0^* = -X_0$. Let us assume that $D$ is a $^*$-derivation. We know that $X_0$ can be replaced with any element in the set $X_0 + i Z(C(\Omega,M_n))_{sa} = \{ X_0 + i f\otimes I_n: f\in C(\Omega)_{sa}\}$, where $I_n$ stands for the unit in $M_n$. The question is whether we can find an element $X_0 + i f\otimes I_n$ satisfying $$\|D\| = \|[X_0 + i f, .]\|\geq \|X_0 + i f\|.$$

When $z$ is a symmetric (or a skew symmetric) operator in $B(H)$, J.G. Stampfli establishes in \cite[Corollary 1]{Stampfli} that \begin{equation}\label{eq norm [z,.] for skew z 0} \|[z,.] \| = 2 \rho (\sigma(z)) = \hbox{diam} (\sigma(z))\leq 2 \|z\|,\end{equation}  and consequently, {if $0\in\sigma(z)$ then}
\begin{equation}\label{eq norm [z,.] for skew z}\|z\| \leq \|[z,.] \| =  \hbox{diam} (\sigma(z))\leq 2 \|z\|.
\end{equation}

We shall require a variant of the previous estimations.

\begin{proposition}\label{p Stampfli on CKMn} Let $D: C(\Omega,M_n) \to C(\Omega,M_n)$ be a $^*$-derivation, where $\Omega$ is a compact Hausdorff space. Then there exists a unique $Z_0\in C(\Omega,M_n)$ satisfying $Z_0^*= -Z_0$, $D(.) = [Z_0,.]$, $\sigma(-i Z_0)\subseteq \mathbb{R}_0^{+}$, and $ \|Z_0\| = \|D\|.$
\end{proposition}

\begin{proof} We have already observed that there exists $Z_1\in C(\Omega,M_n)$ satisfying $Z_1^*= -Z_1$ and $D(.) = [Z_1,.]$. For $Z\in C(\Omega,M_n)$ with $Z^*=-Z$ we set $$\hbox{diam} (\sigma(Z)) := \sup_{t\in \Omega} \hbox{diam} (\sigma(Z(t))).$$

We claim that \begin{equation}\label{eq norm Stampfli} \| D\| = \|  [Z_1,.]\| = \hbox{diam} (\sigma(Z)).
\end{equation} Indeed, for each $X\in C(\Omega,M_n)$ with $\|X\|\leq 1$, it follows from \eqref{eq norm [z,.] for skew z 0} that $$\|[Z_1,X]\| = \sup_{t\in \Omega} \|[Z_1(t),X(t)]\| \leq \sup_{t\in \Omega} \hbox{diam} (\sigma(Z_1(t))) = \hbox{diam} (\sigma(Z_1)).$$ To see the reciprocal inequality, given $\varepsilon>0$, there exists $t_0\in \Omega$ such that $\hbox{diam} (\sigma(Z_1))-\varepsilon < \hbox{diam} (\sigma(Z_1(t_0))).$ Since $[Z_1(t_0), .] : M_n\to M_n$ is a bounded linear operator on a finite dimensional space, we can find a norm-one element $b\in M_n$ such that $\| [Z_1(t_0), b] \|=\| [Z_1(t_0), .] \|= \hbox{diam} (\sigma(Z_1(t_0)))$. Clearly, $\Gamma (b)\in C(\Omega,M_n),$ $\|\Gamma (b) \|\leq 1$ and $$\| [Z_1, .]\| \geq \|[Z_1, \Gamma(b)]\|\geq \|[Z_1(t), \Gamma(b) (t)]\| = \| [Z_1(t_0), b] \|   >\hbox{diam} (\sigma(Z_1))-\varepsilon,$$ which proves the claim.\smallskip

We define a function $\sigma_{min} : \Omega\to \mathbb{C}$, $\sigma_{min} (t) := \lambda\in \sigma(Z_1(t))$, where $\lambda$ is the unique element in $\sigma(Z_1(t))\subseteq i \mathbb{R}$ satisfying $|\lambda| = \min\{|\mu|: \mu \in \sigma (Z_1(t))\}$.\smallskip

Let us recall some notation. Suppose $K_1$ and $K_2$ are non-empty compact subsets of $\mathbb{C}$. The Hausdorff distance between $K_1$ and $K_2$ is defined by $$d_{H} (K_1,K_2) = \max\{ \sup_{t\in K_1} \hbox{dist} (t,K_2), \sup_{s\in K_2} \hbox{dist} (s,K_1) \}.$$ By \cite[Theorem 6.2.1$(v)$]{Aup91} the inequality $$d_{H} (\sigma(a),\sigma(b)) \leq \|a-b\|,$$ holds for all normal elements $a,b$ in a C$^*$-algebra $A$. Applying the above inequality, and the continuity of $-Z_1^*=Z_1(.): \Omega\to M_n$, we shall see that $\sigma_{min} \in C(\Omega)$. Namely, fix $t_0\in \Omega$, $\varepsilon >0$ and an open neighborhood $t_0\in U$ such that $$d_{H} (\sigma(X(t)),\sigma(X(t_0)))\leq \|X(t)-X(t_0)\|<{\varepsilon},$$ for every $t\in U$. Let us write $\sigma(X(t)) = \{\sigma_{min} (t)=\lambda_{1}(t),\lambda_{2}(t),\ldots,\lambda_{n} (t)\}$ and $\sigma(X(t_0)) = \{\sigma_{min} (t_0)=\lambda_{1}(t_0),\lambda_{2}(t_0),\ldots,\lambda_{n} (t_0)\}$ with $-i \sigma_{min} (t)\leq -i \lambda_{2}(t)\leq \ldots\leq -i \lambda_{n} (t)$ and $-i \sigma_{min} (t_0)$ $\leq -i \lambda_{2}(t_0)$ $\leq \ldots\leq -i \lambda_{n} (t_0)$. In this case, for every $t\in U$, there exist $\lambda_{j}(t)$ and $\lambda_k (t_0)$ such that $|\sigma_{min} (t_0) - \lambda_{j}(t)|<\varepsilon$ and $|\sigma_{min} (t) - \lambda_{k}(t_0)|<\varepsilon.$\smallskip

\noindent If $- i \sigma_{min} (t_0) \leq- i \sigma_{min} (t)  \leq -i \lambda_j (t)$ we have $$|\sigma_{min} (t_0) - \sigma_{min} (t)|\leq |\sigma_{min} (t_0) - \lambda_{j}(t)|<\varepsilon.$$ If $- i \sigma_{min} (t) <- i \sigma_{min} (t_0)  \leq -i \lambda_k (t_0)$ we have $$|\sigma_{min} (t_0) - \sigma_{min} (t)|\leq |\sigma_{min} (t) - \lambda_{k}(t_0)|<\varepsilon.$$ Therefore $|\sigma_{min} (t_0) - \sigma_{min} (t)| <\varepsilon$, for every $t\in U$.\smallskip

Clearly, $\sigma_{min}\otimes I_n\in Z(C(\Omega,M_n))$ and $0\in \sigma(Z_1-\sigma_{min}\otimes I_n)$. Applying \eqref{eq norm Stampfli} we conclude that $$ \| D\| = \|  [Z_1-\sigma_{min}\otimes I_n,.]\| = \hbox{diam} (\sigma(Z_1-\sigma_{min}\otimes I_n))= \|Z_1-\sigma_{min}\otimes I_n\|,$$ which proves the desired statement for $Z_0 = Z_1-\sigma_{min}\otimes I_n$.
\end{proof}

\begin{remark}\label{r prop Stampfli for finite dimensional Cstaralgebras}{\rm Let $\displaystyle N=\bigoplus_{1\leq j\leq m}^{ \ell_{\infty}} M_{n_j}$ be an arbitrary finite dimensional C$^*$-algebra (compare \cite[Theorem I.11.2]{Tak}). Let $D: C(\Omega,N) \to C(\Omega,N)$ be a $^*$-derivation, where $\Omega$ is a compact Hausdorff space. Then there exists a unique $Z_0\in C(\Omega,N)$ satisfying $Z_0^*= -Z_0$, $D(.) = [Z_0,.]$, $\sigma(-i Z_0)\subseteq \mathbb{R}_0^{+}$, and $ \|Z_0\| = \|D\|.$ Indeed, we can identify $C(\Omega,N)$ with the $\ell_{\infty}$-sum $\displaystyle \bigoplus_{1\leq j\leq m}^{ \ell_{\infty}} C(\Omega,M_{n_j}).$ It is known that $D (C(\Omega,M_{n_j})) \subseteq C(\Omega,M_{n_j})$ for every $j$ (compare \cite[Lemma 3.3 and its proof]{CaPe2016}). Therefore $D_j = D |_{C(\Omega,M_{n_j})} : C(\Omega,M_{n_j})\to C(\Omega,M_{n_j})$ is a derivation for every $j$, and we identify $D$ with the direct sum of all $D_j$. The desired conclusion follows by applying the above Proposition \ref{p Stampfli on CKMn} to each $D_j$.
}\end{remark}

In accordance with the notation in \cite{CaPe2016}, henceforth, the set of all finite dimensional subspaces of $H$ will be denoted by $ \mathfrak{F} (H).$ This set is equipped with the natural order given by inclusion, and for each $F\in \mathfrak{F} (H)$, $p_{_F}$ will denote the orthogonal projection of $H$ onto $F$. Given a compact Hausdorff space $\Omega$. We set $\widehat{p}_{_F} = \Gamma (p_{_F})$, where $\Gamma : K(H) \to C(\Omega) {\otimes}K(H)$ is the mapping defined before Theorem \ref{t weak-2-local der on CKA}.\smallskip

An \emph{approximate unit} or \emph{identity} in a C$^*$-algebra $A$ is a net $(u_\lambda)\subseteq A$ satisfying $0\leq u_{\lambda} \leq 1$ for every $\lambda$, $u_\lambda\leq u_{\mu}$ for every $\lambda\leq \mu$ and $\lim_{\lambda} ||x - u_{\lambda} x\| = 0$ for each $x$ in $A$. In these conditions, $\lim_{\lambda} \|x - x u_{\lambda}\| = 0$ as well.\smallskip

The following proposition, which has not been explicitly treated in the literature, is all we shall require to deal with the case of weak-2-local derivations on $C(\Omega,K(H))$.

\begin{proposition}\label{p derivations on C(K,K(H))}  Let $D:C(\Omega,K(H))\to C(\Omega,K(H))$ be a derivation, where $H$ is a complex Hilbert space. Let $\tau$ denote the topology of $\Omega$. Then there exists a $\tau$-weak$^*$-continuous, bounded mapping $Z_0:  \Omega\to B(H)$ satisfying $D(X) (t) = [Z_0(t),X(t)]$, for every $X\in C(\Omega,K(H))$. In particular, for each $t$ in $\Omega$ we have $\delta_t D \Gamma \delta_t (A)= \delta_t D (A),$ for every $A\in C(\Omega,K(H))$.
\end{proposition}

\begin{proof} Let us first assume that $D$ is a $^*$-derivation.\smallskip

To simplify the notation we write $C=C(\Omega,K(H))$. Pick an arbitrary $F\in \mathfrak{F} (H)$. Applying \cite[Proposition 2.7]{NiPe2014} we deduce that the mapping $$\widehat{p}_{_F} D \widehat{p}_{_F}|_{\widehat{p}_{_F} C \widehat{p}_{_F} } : \widehat{p}_{_F} C \widehat{p}_{_F} \to \widehat{p}_{_F} C \widehat{p}_{_F}, \ {\widehat{p}_{_F}}A {\widehat{p}_{_F}}\mapsto \widehat{p}_{_F} D (\widehat{p}_{_F}  A \widehat{p}_{_F} ) \widehat{p}_{_F}$$ is a derivation on $\widehat{p}_{_F} C \widehat{p}_{_F}\cong C(\Omega,M_n)$, with $n= \hbox{dim} (F)$. Applying Theorem \ref{l C(KvN) has the inner derivation property or rigidity property} and Proposition \ref{p Stampfli on CKMn} we find a unique $Z_{_F}\in \widehat{p}_{_F} C \widehat{p}_{_F}$ satisfying $Z_{_F}^*= -Z_{_F}$, $\widehat{p}_{_F} D(\widehat{p}_{_F}  A \widehat{p}_{_F}) \widehat{p}_{_F} = [Z_{_F},\widehat{p}_{_F}  A \widehat{p}_{_F}]$, for every $A\in C$,  $\sigma(-i Z_{_F})\subseteq \mathbb{R}_0^{+}$, and $ \|Z_{_F}\| = \| \widehat{p}_{_F} D \widehat{p}_{_F}|_{\widehat{p}_{_F} C \widehat{p}_{_F} } \|\leq \|D\|.$\smallskip

Fix $t$ in $\Omega$. The net $(Z_{_F}(t))_{F\in \mathfrak{F} (H)}\subseteq K(H)\subseteq B(H)$ is bounded, so there exists a subnet $(Z_{_F}(t))_{F\in \mathfrak{F}'}$ and $Z_0(t) \in B(H)$ such that $\|Z_0(t)\|\leq \|D\|$, and $(Z_{_F}(t))_{F\in \mathfrak{F}'}\to Z_0(t)$ in the weak$^*$-topology of $B(H)$.\smallskip

It is not hard to see that the net $(\widehat{p}_{_F})_{F\in \mathfrak{F} (H)}$ is an approximate unit in $C(\Omega,K(H))$. Indeed, let us fix $A$ in $C(\Omega,K(H))$ and $\varepsilon>0$. Applying the continuity of $A$ and the compactness of $\Omega$, we can find a finite open cover $U_1,\ldots,U_m,$ and points $t_1,\ldots,t_m$ such that $t_j\in U_j$ and $\|A(s)-A(t_j)\|<\frac{\varepsilon}{3},$ for each $s\in U_j$. Since $A(t_1),\ldots, A(t_m)\in K(H)$, we can find $F_1\in \mathfrak{F} (H)$ satisfying $\|A(t_j) - p_{_F} A(t_j) p_{_F}\|< \frac{\varepsilon}{3}$, for every $j$ and every $F\in \mathfrak{F} (H)$ with $F\supseteq F_1$. For each $s$ in $\Omega$, there exists $j$ such that $s\in U_j$, and hence $$\|A(s)- p_{_F} A(s) p_{_F}\| \leq \|A(s)- A(t_j)\|+\| A(t_j) - p_{_F} A(t_j) p_{_F}\|+ \|p_{_F} A(t_j) p_{_F}- p_{_F} A(s) p_{_F}\|<\varepsilon,$$ which proves that $\|A- \widehat{p}_{_F} A(s) \widehat{p}_{_F}\| <\varepsilon,$ every $F\in \mathfrak{F} (H)$ with $F\supseteq F_1$.\smallskip

By the continuity of $D$, given $A$ in $C(\Omega,K(H))$, the nets $(\widehat{p}_{_F} A \widehat{p}_{_F})_{F\in \mathfrak{F} (H)}$, and $(\widehat{p}_{_F} D(\widehat{p}_{_F} A \widehat{p}_{_F}) \widehat{p}_{_F})_{F\in \mathfrak{F} (H)}$ converge in norm to $A$ and $D(A)$, respectively. Consequently, for each $t\in \Omega$, $(\widehat{p}_{_F} A \widehat{p}_{_F} (t))_{F\in \mathfrak{F}'}\to A(t)$, and $(\widehat{p}_{_F} D(\widehat{p}_{_F} A \widehat{p}_{_F}) \widehat{p}_{_F} (t))_{F\in \mathfrak{F}'}\to D(A)(t)$ in the norm topology of $K(H)$. Taking weak$^*$-limit in the identity $$\widehat{p}_{_F} D(\widehat{p}_{_F}  A \widehat{p}_{_F}) \widehat{p}_{_F} (t) = [Z_{_F},\widehat{p}_{_F}  A \widehat{p}_{_F}] (t)  =   [Z_{_F} (t) ,\widehat{p}_{_F}  A \widehat{p}_{_F}(t)],$$ we deduce, via Lemma \ref{l weakstar times norm convergent}, that $D(A) (t) = [Z_0 (t), A(t)]$.\smallskip

When $D$ is a general derivation, we write $D = D_1 + i D_2$ with $D_1$ and $D_2$ $^*$-derivations on $C(\Omega,K(H))$. By the arguments above, there exist bounded maps $Z_1,Z_2 : \Omega\to B(H)$ satisfying $D_j(A) (t) = [Z_j(t), A(t)],$ for every $t\in \Omega$, $A\in C(\Omega,K(H))$. Therefore, $D (A) (t) = [Z_1(t)+ i Z_2(t), A(t)],$ for all $t\in \Omega$, $A\in C(\Omega,K(H))$.\smallskip

We shall finally show that $Z_0:  \Omega\to B(H)$ is $\tau$-weak$^*$-continuous. We have already shown that  for each $F\in \mathfrak{F} (H)$ we have $$ \widehat{p}_{_F} D \widehat{p}_{_F}|_{\widehat{p}_{_F} C \widehat{p}_{_F} } (.) = \widehat{p}_{_F} [Z_{0},\widehat{p}_{_F}\  .\ \widehat{p}_{_F}]\widehat{p}_{_F}= \widehat{p}_{_F} [\widehat{p}_{_F} Z_{0} \widehat{p}_{_F}, . ]\widehat{p}_{_F},$$ and hence, by Theorem \ref{l C(KvN) has the inner derivation property or rigidity property}, there exists $Z_{_F} \in C(\Omega, p_{_F} B(H) p_{_F})$ such that  $$\widehat{p}_{_F} Z_{0} \widehat{p}_{_F} - Z_{_F}\in Z(C(\Omega, p_{_F} B(H) p_{_F}) = C(\Omega)\otimes I_{_F}.$$ Therefore, \begin{equation}\label{eq continuity on finite subspaces} \widehat{p}_{_F} Z_{0} \widehat{p}_{_F}\in C(\Omega, p_{_F} B(H) p_{_F})\subseteq C(\Omega, B(H)).
\end{equation}

Pick a normal functional $\phi\in B(H)_*$. Given $\varepsilon>0$. If we identify $B(H)_*$ with the trace-class operators, we can easily find a finite projection $p_{_F}$ with $F\in \mathfrak{F} (H)$ such that $\|\phi - \phi (p_{_F} . p_{_F})\|<\varepsilon$. Combining this fact with \eqref{eq continuity on finite subspaces} we can easily deduce that $\phi\circ Z_0 : \Omega \to \mathbb{C}$ is continuous, which finishes the proof.
\end{proof}

\begin{remark}\label{new remark}{\rm We cannot assure, in general that the mapping $Z_0: \Omega\to B(H)$ is $\tau$-norm continuous. Let us take an infinite dimensional Hilbert space $H$. Let $p_n$ be a sequence of mutually orthogonal rank one projections in $B(H)$. The von Neumann subalgebra $C$ of $B(H)$ generated by the $p_n$'s is C$^*$-isomorphic to $\ell_{\infty}$.\smallskip

The set $\Omega:= \{ a\in C\cong \ell_{\infty} : \|a\|\leq 1 \}$ is weak$^*$-closed and hence $(\Omega,\tau =\hbox{weak}^*)$ is a compact Hausdorff space. Let $Z_0 : \Omega \to C\subset B(H)$ be the identity mapping, which is $\tau$-weak$^*$-continuous and bounded. Clearly, $Z_0$ is not weak$^*$-norm continuous.\smallskip

It is known that a bounded net $(a_{\lambda})$ in $\ell_{\infty}$ converges in the weak$^*$-topology to an element $a_0$ if and only if for each natural $n$, $(|a_{\lambda} (n)- a_0 (n)|)\to 0$. Suppose $(a_\lambda)\subset \Omega$, $(a_\lambda)\to a_0$ in $\Omega$. Since $a_{\lambda} \to a_0$ in the weak$^*$-topology of $C$, it is not hard to see that $(a_{\lambda}-a_0) (a_{\lambda}-a_0)^* \to 0$ in the weak$^*$-topology of $C$, and hence $(a_{\lambda}-a_0)(a_{\lambda}-a_0)^* \to 0$ in the weak$^*$-topology of $B(H)$.\smallskip

We claim that, for each finite rank projection $p\in B(H)$, the mappings $Z_0 p,$ and $p Z_0$ both are $\tau$-norm continuous. Indeed, let $(a_{\lambda})\to a_0$ in $\Omega$ (i.e. in the weak$^*$-topology of $B(H)$). The arguments given in the above paragraph show that $(a_{\lambda}-a_0)(a_{\lambda}-a_0)^* \to 0$ in the weak$^*$-topology of $B(H).$ Since the operator $U_{p}: B(H)\to B(H)$, $x\mapsto pxp$ has finite rank, we deduce that $\|p (a_{\lambda}-a_0)\|^2 =\| (a_{\lambda}-a_0) p\|^2 =  \| p (a_{\lambda}-a_0) (a_{\lambda}-a_0)^* p\|\stackrel{\lambda}{\to} 0$,  which proves the claim.\smallskip

We shall finally show that for each $X\in C(\Omega, K(H))$, $Z_0 X$ and $X Z_0$ both lie in $C(\Omega, K(H))$. We may assume, without loss of generality, that $\|X\|\leq 1$. Let $(a_{\lambda})\to a_0$ in $\Omega$, and let $\varepsilon>0$. By the continuity of $X$, there exists $\lambda_0\in \Lambda$ such that $\|X(a_{\lambda})-X(a_0)\|<\frac{\varepsilon}{3}$ for every $\lambda>\lambda_0$. Since $X(a_0)\in K(H)$, we can find a finite rank projection $p\in B(H)$ such that $\|X(a_0)-p X(a_0) \|<\frac{\varepsilon}{3}$. By the $\tau$-norm continuity of $Z_0 p$, there exists $\lambda_1\geq \lambda_0$ such that $\|(Z_0(a_{\lambda}) - Z_0(a_0)) p\|<\frac{\varepsilon}{3}$, for all $\lambda\geq \lambda_1$. Therefore for all $\lambda \geq \lambda_1$ we have  $$\|(Z_0 X)(a_{\lambda}) - (Z_0 X)(a_0)\| \leq \| Z_0(a_{\lambda}) (X(a_{\lambda})-X(a_0))\| +\| (Z_0(a_{\lambda}) - Z_0(a_0)) X (a_0) \|$$ $$\leq \frac{\varepsilon}{3} +\| (Z_0(a_{\lambda}) - Z_0(a_0)) (X (a_0)- p X (a_0)) \|  +\| (Z_0(a_{\lambda}) - Z_0(a_0)) p X (a_0) \| $$ $$\leq \frac{\varepsilon}{3} + 2 \frac{\varepsilon}{3}+\| (Z_0(a_{\lambda}) - Z_0(a_0)) p  \|\leq \varepsilon. $$ This shows that $Z_0 X\in C(\Omega, K(H)).$ The statement for $X Z_0$ follows similarly. Then $D: C(\Omega, K(H)) \to C(\Omega, K(H))$, $D(X) = [Z_0,X] = Z_0 X -X Z_0$ is a derivation on $C(\Omega, K(H))$.
}\end{remark}

Though it is not true that every derivation on $C(\Omega,K(H))$ is inner, and hence Theorem \ref{t weak-2-local der on CKA} cannot be applied, we can extend our study to weak-2-local derivations on $C(\Omega,K(H))$.

\begin{theorem}\label{t weak-2-local der on CKH} Let $H$ be a complex Hilbert space. Then every weak-2-local derivation $\Delta:C(\Omega,K(H))\to C(\Omega,K(H))$ is a {\rm(}linear{\rm)} derivation.
\end{theorem}

\begin{proof} Let us fix $t$ in $\Omega$ and $\phi \in K(H)^*$. For the functional $\phi\otimes \delta_{t}\in C(\Omega,K(H))^*$, $A$, $\Gamma \delta_t (A)$ in $C(\Omega,K(H))$, there exists a derivation $D=D_{A,\Gamma \delta_t (A), \phi\otimes \delta_{t}}$ on $C(\Omega,K(H))$ such that $(\phi\otimes \delta_{t}) \Delta(A) = \phi\otimes \delta_{t} D(A)$ and $(\phi\otimes \delta_{t}) \Delta \Gamma \delta (A) = (\phi\otimes \delta_{t}) D \Gamma \delta_t (A)$. Since, by Proposition \ref{p derivations on C(K,K(H))}, $\delta_t D \Gamma \delta_t (A)= \delta_t D (A),$ we deduce that \begin{equation}\label{eq new deltatDeltaGamma deltat} \delta_t \Delta \Gamma \delta_t = \delta_t \Delta.\end{equation}

Arguing as in the proof of Theorem \ref{t weak-2-local der on CKA}, we deduce that $\delta_t \Delta \Gamma : K(H) \to K(H)$ is a weak-2-local derivation. Theorem 3.2 in \cite{CaPe2016} assures that  $\delta_t \Delta \Gamma$ is a linear derivation. Applying the identity in \eqref{eq new deltatDeltaGamma deltat}, and following the same arguments given at the end of the proof of Theorem \ref{t weak-2-local der on CKA}, we obtain $\delta_t \Delta (X+Y)= \delta_t \Delta (X)+ \delta_t \Delta (Y)$, for every $X,Y$ in $C(\Omega,K(H)),$ which proves the first statement.\end{proof}

The machinery developed in previous results reveals a pattern which is stated in the next result, whose proof has been outlined in Theorems \ref{t weak-2-local der on CKA} and \ref{t weak-2-local der on CKH}.

\begin{theorem}\label{t pattern} Let $\Omega$ be a compact Hausdorff space. Suppose $A$ is a C$^*$-algebra satisfying the following hypothesis:\begin{enumerate}[$(a)$] \item Every weak-2-local derivation on $A$ is a {\rm(}linear{\rm)} derivation;
\item For each derivation $D: C(\Omega,A)\to C(\Omega,A)$ the identity $\delta_t D \Gamma \delta_t (A)= \delta_t D (A),$ holds for every $t\in \Omega$, $A\in C(\Omega,A)$.
\end{enumerate} Then every weak-2-local derivation on $C(\Omega,A)$ is a {\rm(}linear{\rm)} derivation.$\hfill\Box$
\end{theorem}

Proposition \ref{p derivations on C(K,K(H))} above shows that $K(H)$ satisfies the hypothesis $(b)$ in the previous theorem. We recall that every compact C$^*$-algebra $B$ is C$^*$-isomorphic to the $c_0$-sum $(\bigoplus_{i \in I} K(H_i))_{c_0}$, where each $H_i$ is a complex Hilbert space (see \cite{Al}). Let us define a particular approximate unit in $B$. Let $\mathcal{F} (I)$ denote the finite subsets of $I$. Let $\Lambda$ be the set of all finite tuples of the form $(F_i )_{i\in J}=(F_{i_1},\ldots, F_{i_k})$, where $J=\{i_1,\ldots, i_k\}\in \mathcal{F}(I)$, and $F_{i_j}\in\mathfrak{F} (H_{{i_j}})$. We shall say that $(F_{i})_{i\in J_1}\leq (G_{i})_{i\in J_{2}}$ if $J_1\subseteq J_2$ and $F_i\subseteq G_i$ for every $i\in J_1$. We set $\displaystyle p_{(F_{i})_{i\in J}} := \sum_{i\in J} p_{_{F_i}}\in B$. The net $(\displaystyle p_{(F_{i})_{i\in J}})_{_{(F_{i})_{i\in J}\in \Lambda}}$ is an approximate unit in $B$. When in the proof of Proposition \ref{p derivations on C(K,K(H))} the approximate unit $(\widehat{p}_{_F})_{F\in \mathfrak{F} (H)}$ is replaced with $(\displaystyle p_{(F_{i})_{i\in J}})_{_{(F_{i})_{i\in J}\in \Lambda}}$, and Proposition \ref{p Stampfli on CKMn} is substituted for Remark \ref{r prop Stampfli for finite dimensional Cstaralgebras}, the arguments remain valid to prove the following:

\begin{proposition}\label{p derivations on C(K,compact)} Let $D:C(\Omega,B)\to C(\Omega,B)$ be a derivation, where $\Omega$ is a compact Hausdorff space and $B$ is a compact C$^*$-algebra. Then there exists a bounded mapping $Z_0:  \Omega\to B^{**}$ satisfying $D(X) (t) = [Z_0(t),X(t)]$, for every $X\in C(\Omega,B)$. In particular, for each $t$ in $\Omega$ we have $\delta_t D \Gamma \delta_t (A)= \delta_t D (A),$ for every $A\in C(\Omega,B)$.$\hfill\Box$
\end{proposition}

Finally, combining Proposition \ref{p derivations on C(K,compact)} with  Proposition 3.4 in \cite{CaPe2016} and Theorem \ref{t pattern} we culminate the study of weak-2-local derivation on the C$^*$-algebra of all continuous functions on a compact Hausdorff space with values on a compact C$^*$-algebra.

\begin{theorem}\label{t weak-2-local der on compact Cstar} Let $B$ be a compact C$^*$-algebra, and let $\Omega$ be a compact Hausdorff space. Then every weak-2-local derivation $\Delta:C(\Omega,B)\to C(\Omega,B)$ is a {\rm(}linear{\rm)} derivation. In particular, every 2-local derivation on $C(\Omega,B)$ is a {\rm(}linear{\rm)} derivation. $\hfill\Box$
\end{theorem}

We have already commented that, by a result of Sh. Ayupov and K. Kudaybergenov every 2-local derivation on a von Neumann algebra $M$ is a derivation \cite{AyuKuday2014}, and the problem whether the same statement remains true for general C$^*$-algebras remains open. We can throw new light onto the study of 2-local derivations on C$^*$-algebras. Theorem \ref{t pattern} above admits the following corollary.

\begin{corollary}\label{c pattern} Let $\Omega$ be a compact Hausdorff space. Suppose $B$ is a C$^*$-algebra satisfying the following hypothesis:\begin{enumerate}[$(a)$] \item Every  2-local derivation on $B$ is a {\rm(}linear{\rm)} derivation;
\item For each derivation $D: C(\Omega,B)\to C(\Omega,B)$ the identity $\delta_t D \Gamma \delta_t (A)= \delta_t D (A),$ holds for every $t\in \Omega$, $A\in C(\Omega,B)$.
\end{enumerate} Then every 2-local derivation on $C(\Omega,B)$ is a {\rm(}linear{\rm)} derivation.$\hfill\Box$
\end{corollary}

Let $M$ be a von Neumann algebra and let $\Omega$ be a compact Hausdorff space. The previously mentioned theorem of Ayupov and Kudaybergenov assures that $M$ satisfies hypothesis $(a)$ in the above corollary. Theorem \ref{l C(KvN) has the inner derivation property or rigidity property} implies that $C(\Omega,M)$ also satisfies hypothesis $(b)$, we therefore obtain the following strengthened version of Corollary \ref{c Ayupov Arzikulov}.

\begin{corollary} Let $M$ be a von Neumann algebra and let $\Omega$ be a compact Hausdorff space. Then every 2-local derivation on $C(\Omega,M)$ is a {\rm(}linear{\rm)} derivation.$\hfill\Box$
\end{corollary}

\end{document}